\documentclass[9pt,shortpaper,twoside,web]{ieeecolor}
\usepackage{generic}
\usepackage{cite}
\usepackage{amssymb,amsfonts}
\usepackage{algorithmic}
\usepackage{graphicx}
\usepackage{textcomp}
\def\BibTeX{{\rm B\kern-.05em{\sc i\kern-.025em b}\kern-.08em
    T\kern-.1667em\lower.7ex\hbox{E}\kern-.125emX}}
\usepackage[dvipsnames]{xcolor}
\usepackage{color, soul}
\sethlcolor{red}

\usepackage{textcomp}
\def\BibTeX{{\rm B\kern-.05em{\sc i\kern-.025em b}\kern-.08em
    T\kern-.1667em\lower.7ex\hbox{E}\kern-.125emX}}
\usepackage{float}
\usepackage[tbtags]{amsmath}
\usepackage[tbtags]{mathtools}
\usepackage{accents}
\usepackage{multicol}
\usepackage[ruled,vlined]{algorithm2e}
\usepackage{cite}
\usepackage{pgfplots}

\usepackage{array,multirow}
\usepackage{hhline}
\usepackage{arydshln}

\usepackage[hidelinks=true]{hyperref}

\newtheorem{theorem}{Theorem}
\newtheorem{corollary}{Corollary}
\newtheorem{lemma}{Lemma}

\newtheorem{assumption}{Assumption}
\newtheorem{prop}{Proposition}
\newtheorem{problem}{Problem}
\newtheorem{alg}{Algorithm}
\newtheorem{remark}{Remark}

\newtheorem{definition}{Definition}

\makeatletter
\def\fnum@figure{\textcolor{subsectioncolor}{\sf Fig.~\thefigure}}
\def\fnum@table{\textcolor{subsectioncolor}{\sf TABLE~\thetable}}
\makeatother

\begin{document}
\title{Prescribed-time Control for Perturbed Euler-Lagrange Systems with Obstacle Avoidance}
\author{Amir Shakouri, \IEEEmembership{Member, IEEE}, and Nima Assadian, \IEEEmembership{Senior Member, IEEE}
\thanks{The authors are with the Department of Aerospace Engineering, Sharif University of Technology, Tehran, Iran (e-mail: \href{mailto:a_shakouri@outlook.com}{a$\_$shakouri@outlook.com}; \href{mailto:assadian@sharif.edu}{assadian@sharif.edu}).}}

\maketitle

\begin{abstract}
This paper introduces a class of time-varying controllers for Euler-Lagrange systems such that the convergence occurs at an arbitrary finite time, independently of initial conditions, and free of chattering. The proposed controller is based on a mapping technique and is designed in two steps: First, a conventional (obstacle avoidance) asymptotically stable controller is specified for the nominal system; then, by a simple substitution, a prescribed-time (obstacle avoidance) controller is achievable for the perturbed system. It is proved that the proposed scheme is uniformly prescribed-time stable for unperturbed systems and prescribed-time attractive for perturbed systems as it rejects matched disturbances with unknown upper bounds without disturbance observation. As an example, a two-link robot manipulator is considered for numerical simulations.
\end{abstract}

\section{Introduction}
\label{sec:I}
\IEEEPARstart{T}{he} finite-time control of nonlinear systems has been a challenging field of research in the past two decades in which the robustness and the smoothness issues have been drawn much attention. Specifically, for robotic applications, various endeavors have been made to enhance the performance of the system response. However, in finite-time control techniques, the convergence time is dependent on the initial condition, and that is why fixed-time control has recently been an attractive topic; it can provide an upper bound for the time of stabilization. The prescribed-time control goes one step further and makes it possible for the user to arbitrarily specify the convergence time by constructing a time-varying high-gain controller capable of rejecting matched disturbances. 

The possibility of commanding a highly perturbed nonlinear system to converge at an arbitrary time is pivotal in many crucial missions related to robotic systems, aircraft, and spacecraft where high precision is required while traditional control methods may fail according to the modeling inaccuracies. For instance, in the control of a quadrupedal robot, where the system is subject to unknown time-varying loads and constraints \cite{iqbal2020provably}, the prescribed-time scheme can be an effective solution for the control of system elements without using disturbance observers and system identification. In complex environments where several robots are cooperating \cite{petersen2019review}, being able to establish a definite time constraint on agents' coordination may cause safe operation in sensitive tasks. As another example, the attitude control of micro aerial vehicles, whose utilization is often limited by ineffective controllers in turbulent environments \cite{mohamed2014attitude}, can be accomplished by employing a control method capable of coping with unknown disturbances and stabilize the vehicle in a specified time.

The non-smooth feedback control \cite{in2,in3} and the terminal sliding mode control \cite{in5,in6,in7} are the most popular methods for finite-time control of nonlinear systems. To deal with uncertainties, different approaches are used in the literature for achieving an adaptive finite-time control and tracking scheme \cite{in8,in9,in11}. The fixed-time control is first proposed for perturbed linear systems by Polyakov \cite{in12}, and since then, many investigations are carried out to push the boundaries of this area. The non-singular terminal sliding mode control is proposed in \cite{in13} for a class of second-order nonlinear systems with matched disturbances. An output feedback scheme for perturbed double integrator systems is addressed in \cite{in14} while the stabilization of high-order integrator systems with mismatched disturbances is studied in \cite{in15}. The fixed-time attitude stabilization problem in the presence of disturbances, faults, and saturation is investigated in \cite{in16}. 

Time-varying approaches for prescribed-time stabilization of perturbed nonlinear systems are proposed by Song et al. \cite{in18,in19}. The behavior of time-varying methods under non-vanishing uncertainties is studied in \cite{in20}. These time-varying techniques are called \textit{prescribed-time controllers} (PTCs) since the time of convergence can be arbitrarily specified by the user. A similar approach called the generalized time transformation method is proposed in \cite{tran2020finite,krishnamurthy2020dynamic}. A class of PTCs with linear state/output feedback has also been studied in \cite{zhou2021prescribed}. Moreover, a class of PTCs with linear decay rate is proposed in \cite{shakouri2021prescribed} for normal form nonlinear systems under matched disturbance and uncertain input gain. For Euler-Lagrange systems, a proportional-integral PTC is recently studied in \cite{cui2021prescribed}. 

\subsection{Contribution of the paper}

This paper addresses a prescribed-time control method for Euler-Lagrange systems using a mapping technique that can be used for a wide range of robotic systems such as manipulators and spacecraft. The proposed PTCs in \cite{in18,in19,in20,tran2020finite,krishnamurthy2020dynamic,zhou2021prescribed,shakouri2021prescribed,cui2021prescribed} can be considered particular cases of a more general method in the current paper. The proposed scheme is constructed by two steps: (1) A conventional controller is designed for the nominal unperturbed system to result in an asymptotically stable closed-loop response, which is called the \textit{infinite-time controller} (ITC); then, (2) if some mild conditions are satisfied, by a simple substitution, a PTC is obtained that is \textit{uniformly prescribed-time stable} for the nominal system and \textit{prescribed-time attractive} for the perturbed system. The main idea is based on mapping the position vector by the use of time scaling functions of class $\mathcal{K}$ such that the position history of an asymptotically stable closed-loop response (under the ITC), from zero to infinity, are squeezed and mapped onto a prescribed time interval, from zero to an arbitrary time. The mapped trajectories are indeed the closed-loop responses of the system utilizing the PTC. Therefore, the path taken by an Euler-Lagrange system under the PTC is precisely the same path with the specified ITC, but happening in a different time interval. As a result, the proposed PTC is an obstacle avoidance controller (for fixed obstacles) if an obstacle avoidance ITC is designed for the system in the first step. A simple corollary of the proposed scheme shows how a gain scheduled proportional-derivative (PD) controller with gravity compensation can be a PTC for robotic systems. It is shown that the proposed time-varying controller can reject all bounded disturbances even when the upper bounds are unknown. Moreover, an output assessment method is presented by which any output of the system under the PTC, such as a Lyapunov function, can be corresponded to another output of the system when an ITC is used. A two degrees-of-freedom (2-DOF) robot manipulator has been considered for numerical simulations.

\subsection{Notation}

Let $\mathbb{M}^{m,n}$ denote the space of $m \times n$ real matrices and $\mathbb{M}^n$ its square analog. In addition, let $\mathbb{R}^n$ denotes the space of $n$-dimensional real vectors. The $n$-dimensional identity matrices is denoted by $\mathbb{I}_n$. The $ij$th entry (resp., $i$th entry) of matrix $M\in\mathbb{R}^{n\times m}$ (resp., vector $r\in \mathbb{R}^n$) is referred to by $M_{ij}$ (resp., $r_i$). For matrix $M\in\mathbb{M}^n$ we denote by $M^{-1}$ its inverse (if it exists). An inverse function is denoted by $f^{-1}(\cdot)$ for function $f(\cdot)$ (if the inverse exists). The symbol $\|\cdot\|$ denotes the 2-norm for vectors and matrices.

\section{Preliminaries}
\label{sec:II}

This section introduces all the preliminary formulations required before presenting the main results. First, The nominal and perturbed systems are formulated, the definitions are introduced, and the main problem of the paper is discussed. Finally, the mapping strategy of the paper is presented to be used in the subsequent results. 

\subsection{Basic Formulations and Problem Statement}
\label{subsec:II-C}

Consider the Euler-Lagrange system in the following second-order differential equation form \cite{spong}: 

\begin{equation}
\label{eq:1}
M(q)\ddot{q}+C(\dot{q},q)\dot{q}+g(q)=u
\end{equation}
where $q\in\mathbb{R}^n$ and $u\in\mathbb{R}^n$ denote the system variables and the control input vector, respectively. The matrix valued functions $M(\cdot):\mathbb{R}^n\rightarrow\mathbb{M}^n$ and $C(\cdot,\cdot):\mathbb{R}^n\times\mathbb{R}^n\rightarrow\mathbb{M}^n$ depend on the system characteristics such that $M(q)$ is positive-definite, $C(\dot{q},q)$ is linear in $\dot{q}$, and $\dot{M}(q)-2C(\dot{q},q)$ is skew-symmetric. 

Equation \eqref{eq:1} can be modified in order to contain an unmodeled time-dependent disturbance term $d(t):[0,\infty)\rightarrow\mathbb{R}^n$ as follows: 
\begin{equation}
\label{eq:2}
M(q)\ddot{q}+C(\dot{q},q)\dot{q}+g(q)=d(t)+u
\end{equation}
The disturbance $d(t)$ can be any bounded perturbation forces/torques, such as unmodeled dynamics, parametric uncertainties, etc., for which no upper bounds need to be known. 

\begin{assumption}
\label{ass:d}
Suppose the magnitude of the unmodeled disturbance $d(t)$ is bounded within every finite time interval, i.e., for every $\tau>0$, there exists (unknown) $\bar{d}>0$ such that $\|d(t)\|<\bar{d}$ for all $t\in[0,\tau)$. In other words, $d(t)$ has no vertical asymptotes.
\end{assumption}

The term \textit{infinite-time} (against the terms finite, fixed, or prescribed-time) is frequently used in this paper as an adjective to emphasize that a concept or a function is viewed from $t=t_0$ to $\infty$. For example, an \textit{infinite-time controller} (ITC) is a conventional controller that actively influences the system from $t=t_0$ to $\infty$.

Consider the following definitions about different notions of infinite-time stability:

\begin{definition}[\cite{khalil2002nonlinear}]
\label{def:0}
For a nonautonomous system as
\begin{equation}
\label{eq:def1}
\dot{x}=a(x,t)
\end{equation}
where $x\in\mathbb{R}^{n_x}$, $a(\cdot,\cdot):\mathbb{R}^{n_x}\times[0,\infty)\rightarrow\mathbb{R}^{n_x}$, and $a(0,t)=0$, an equilibrium state $x=0$ is called
\begin{enumerate}
\item \textit{stable (uniformly stable)}, if for every $\varepsilon>0$ there exists $\delta(\varepsilon,t_0)>0$ (respectively, $\delta(\varepsilon)>0$ independent of $t_0$) such that for the trajectories of \eqref{eq:def1}, if $\|x(t_0)\|<\delta(\varepsilon,t_0)$, then  $\|x(t)\|<\varepsilon$ for all $t\in[t_0,\infty)$.
\item \textit{attractive (globally attractive)}, if for every trajectory of \eqref{eq:def1} with  initial conditions $x(t_0)$ inside a neighborhood of equilibrium (respectively, with every initial conditions) we have $\lim_{t\rightarrow\infty}x(t)=0$.
\item \textit{uniformly attractive (globally uniformly attractive)}, if it is attractive (globally attractive) for all $t_0\geq0$. 
\item \textit{(globally) asymptotically stable}, if it is stable and (globally) attractive.
\item \textit{(globally) uniformly asymptotically stable}, if it is uniformly stable and (globally) uniformly attractive.
\item \textit{exponentially stable (globally exponentially stable)}, if there exist $\varrho>0$ and $\lambda>0$ such that for every trajectory of \eqref{eq:def1} with initial conditions inside a neighborhood of equilibrium (respectively, with every initial condition) we have $\|x(t)\|<\varrho\|x(t_0)\|\exp(-\lambda (t-t_0))$ for all $t\in[t_0,\infty)$.
\end{enumerate}
\end{definition}

The definitions of finite and fixed-time stability can be stated as follows:

\begin{definition}[\cite{polyakov2020generalized}]
\label{def:0p}
For a system as \eqref{eq:def1}, an equilibrium state $x=0$ is called
\begin{enumerate}
\item \textit{finite-time attractive (globally finite-time attractive)}, if it is attractive and for every trajectory of \eqref{eq:1} with initial conditions inside a neighborhood of equilibrium (respectively, with every initial conditions) there exists $t^*(t_0,x(t_0))>0$ such that for all $t\in[t_0+t^*,\infty)$ we have $x(t)=0$, i.e., the convergence occurs at a finite time $t^*(t_0,x(t_0))<\infty$.
\item \textit{uniformly finite-time attractive (globally uniformly finite-time attractive)}, if it is finite-time attractive (globally finite-time attractive) such that $t^*(t_0,x(t_0))<\infty$ for all $t_0\geq0$.
\item \textit{(globally) finite-time stable}, if it is stable and (globally) finite-time attractive.
\item \textit{(globally) uniformly finite-time stable}, if it is uniformly stable and (globally) uniformly finite-time attractive.
\item \textit{(globally) fixed-time attractive}, if it is (globally) uniformly finite-time attractive and there exists a $t^*_{\mathrm{max}}>0$ independent of $t_0$ and $x(t_0)$ such that $t^*(t_0,x(t_0))<t^*_{\mathrm{max}}$.
\item \textit{(globally) fixed-time stable}, if it is stable and (globally) fixed-time attractive.
\item \textit{(globally) uniformly fixed-time stable}, if it is uniformly stable and (globally) fixed-time attractive.
\end{enumerate}
\end{definition}

The definition of prescribed-time stability can be expressed as follows, which is used frequently in this paper. 

\begin{definition}
\label{def:0pp}
For a nonautonomous system with a user-defined parameter $\tau>0$ as
\begin{equation}
\label{eq:def2}
\dot{x}=a(x,t,\tau)
\end{equation}
where $x\in\mathbb{R}^{n_x}$, $a(\cdot,\cdot,\cdot):\mathbb{R}^{n_x}\times[0,\infty)\times(0,\infty)\rightarrow\mathbb{R}^{n_x}$, and $a(0,t,\tau)=0$, an equilibrium state $x=0$ is called
\begin{enumerate}
\item \textit{prescribed-time attractive (globally prescribed-time attractive)}, if independent of $t_0$ for every trajectory of \eqref{eq:1} with initial conditions inside a neighborhood of equilibrium (respectively, with every initial conditions) and every $\tau>0$ we have $x(t)=0$ for all $t\in[t_0+\tau,\infty)$, i.e., the convergence occurs at a user-defined finite time $t_0+\tau$.
\item \textit{(globally) prescribed-time stable}, if it is stable and (globally) prescribed-time attractive.
\item \textit{(globally) uniformly prescribed-time stable}, if it is uniformly stable and (globally) prescribed-time attractive.
\end{enumerate}
\end{definition}

The problem statement of this paper can be expressed as follows:

\begin{problem}
\label{prob:1}
Find a function $h(\cdot,\cdot,\cdot):\mathbb{R}^n\times\mathbb{R}^n\times[0,\infty)\rightarrow\mathbb{R}^n$ such that the closed-loop solution of systems \eqref{eq:1} under a controller of the form 
\begin{equation}
\label{eq:3}
u=h(\dot{q},q,t)
\end{equation}
be (globally) uniformly prescribed-time stable and the closed-loop solution of system \eqref{eq:2} under the same controller be (globally) prescribed-time attractive.
\end{problem}

The solution of Problem \ref{prob:1} is a PTC, which is based on a mapping strategy in this paper. This mapping technique is discussed in the following subsection.

\subsection{Mapping Strategy}
\label{subsec:II-B}

This subsection provides preliminary information about the proposed mapping strategy. 

\begin{definition}
\label{def:1}
Consider the following definitions:
\begin{enumerate}
\item A continuous function $\kappa(\cdot):[0,\tau)\rightarrow[0,\infty)$ is said to belong to class $\mathcal{K}$ (or $\kappa\in\mathcal{K}(\tau)$) if it is strictly increasing subject to $\lim_{t\rightarrow0^+}\kappa(t)=0$ and $\lim_{t\rightarrow\tau^-}\kappa(t)=\infty$. This class is a special surjective form of a more general case used in the literature under the same name (see Definition 4.2 in \cite{khalil2002nonlinear}).
\item A class $\mathcal{K}$ function $\kappa(t)$ belongs to class $\mathcal{K}_1\subset\mathcal{K}$ if $\dot{\kappa}(0)=1$ and $\ddot{\kappa}(t)\geq0$ for all $t\in[0,\tau)$.
\item A continuous function $\mu(\cdot):[0,\infty)\rightarrow[0,\tau)$ is said to belong to class $\mathcal{M}$ (or $\mu\in\mathcal{M}(\tau)$) if its inverse function is class $\mathcal{K}$ (or $\mu^{-1}\in\mathcal{K}(\tau)$). Therefore, $\mu$ is a continuous increasing function subject to $\lim_{t\rightarrow0^+}\mu(t)=0$ and $\lim_{t\rightarrow\infty}\mu(t)=\tau$.
\item A class $\mathcal{M}$ function $\mu(t)$ belongs to class $\mathcal{M}_1\subset\mathcal{M}$ if its inverse function is of class $\mathcal{K}_1$, i.e., $\dot{\mu}(0)=1$ and $\ddot{\mu}(t)<0$ for all $t\in[0,\infty)$. 
\end{enumerate}
\end{definition}

The following lemma can be evaluated directly after considering the above definitions. 

\begin{lemma}
\label{lem:1}
Let $\dot{\mu}(t)\coloneqq d\mu(t)/dt$, $\ddot{\mu}(t)\coloneqq d^2\mu(t)/dt^2$, $\kappa^{\prime}(\mu(t))\coloneqq d\kappa(\mu)/d\mu$, and $\kappa^{\prime\prime}(\mu(t))\coloneqq d^2\kappa(\mu)/d\mu^2$. Then, the following statements hold for any functions $\kappa\in\mathcal{K}(\tau)$ and $\mu=\kappa^{-1}\in\mathcal{M}(\tau)$:
\begin{enumerate}
\item $\dot{\mu}(t)=1/\kappa^{\prime}(\mu(t))$ and $\kappa^{\prime\prime}(\mu(t))\dot{\mu}^2(t)+\kappa^{\prime}(\mu(t))\ddot{\mu}(t)=0$.
\item $\kappa^{\prime}(\cdot):[0,\tau)\rightarrow[0,\infty)$, $\kappa^{\prime\prime}(\cdot):[0,\tau)\rightarrow\mathbb{R}$, and $\lim_{t\rightarrow\tau^-}\kappa^{\prime}(t)=\lim_{t\rightarrow\tau^-}\kappa^{\prime\prime}(t)=\infty$.
\item $\dot{\mu}(\cdot):[0,\infty)\rightarrow[0,\infty)$, $\ddot{\mu}(\cdot):[0,\infty)\rightarrow\mathbb{R}$, and $\lim_{t\rightarrow\infty}\dot{\mu}(t)=-\lim_{t\rightarrow\infty}\ddot{\mu}(t)=0$. 
\item The sum of two class $\mathcal{K}$ functions belongs to class $\mathcal{K}$, and the sum of two class $\mathcal{M}$ functions divided by $2$ belongs to class $\mathcal{M}$.
\item Function $\dot{\mu}^\alpha(t)$ for $\alpha\geq1$ is the derivative of a class $\mathcal{M}$ function. 
\item Any exponential function of the form $\beta e^{-\alpha t}$ with arbitrary $\alpha>0$ is the derivative of a class $\mathcal{M}(\tau)$ function if and only if $\beta=\alpha\tau$.
\end{enumerate}
\end{lemma}
\begin{proof}
Item (1) is obtainable by taking the first and second time derivatives of $\kappa(\mu(t))=t$. Item (5) is provable by considering $\dot{\mu}^\alpha(t)\leq\dot{\mu}(t)$ after some time $t_0\in[0,\infty)$, and therefore, $\int_{t_0}^\infty\dot{\mu}^\alpha(t)dt\leq\int_{t_0}^\infty\dot{\mu}(t)dt<\infty$. The rest of the items are easily provable and their proofs are omitted here. 
\end{proof}

The trajectories of a dynamic system defined in the time domain from $t=t_0$ to $\infty$ can be mapped onto a finite time-scale from $t=t_0$ to $t_0+\tau$ by a bijective mapping function, such as $\kappa\in\mathcal{K(\tau)}$ which maps the elements of $[t_0,\infty)$ onto $[t_0,t_0+\tau)$. 

Consider the following form of an autonomous second-order system which operates from $t=t_0$ to $\infty$ at which a unique solution exists corresponding to any initial condition: 
\begin{equation}
\label{eq:6}
\ddot{q}+v(\dot{q},q)=0
\end{equation}
where $v(\cdot,\cdot)$ is an $n$-dimensional vector-valued function. Suppose that it is desired to map the trajectories of \eqref{eq:6} onto a finite interval such that the values of $q$ remain unchanged but happening at a different time. Define $\Delta t=t-t_0$ and $\eta(t)=\mu(\Delta t)+t_0$ (knowing that $d\eta=d\mu$) such that $\mu\in\mathcal{M}(\tau)$. Let $p$ denote the variable vector of the new system and consider $p(\eta(t))=q(t)$. Then, the derivatives in \eqref{eq:6} can be substituted as $\dot{\mu}dp/d\eta=\dot{q}(t)$ and $\ddot{\mu}dp/d\eta+\dot{\mu}^2d^2p/d\eta^2=\ddot{q}(t)$ (the argument of $\dot{\mu}$ and $\ddot{\mu}$ is $\Delta t$, and the argument of $p$ is $\eta$). From the derivative rules of inverse functions, $\dot{\mu}=1/{\kappa}^\prime(\mu)$ and $\ddot{\mu}=-{\kappa}^{\prime\prime}(\mu)/{\kappa^{\prime}}^3(\mu)$ where $\kappa\in\mathcal{K}$. Now, considering $p\in\mathbb{R}^n$ and $\eta\in[t_0,t_0+\tau)$ as the new variables, and defining $\dot{p}=dp/d\eta$ and $\ddot{p}=d^2p/d\eta^2$ the following expression is obtainable:
\begin{equation}
\label{eq:7}
\ddot{p}+w\left(\dot{p},p,\eta\right)=0
\end{equation}
where defining $\Delta \eta=\eta-t_0$
\begin{equation}
\label{eq:8}
w(\dot{p},p,\eta)=-\frac{\ddot{\kappa}(\Delta \eta)}{\dot{\kappa}(\Delta \eta)}\dot{p}+\dot{\kappa}^2(\Delta \eta)v(\dot{p}/\dot{\kappa}(\Delta \eta),p)
\end{equation}
The trajectories of system \eqref{eq:6}, corresponding to the initial values of $q(t_0)$ and $\dot{q}(t_0)$, are squeezed and mapped onto the trajectories of system \eqref{eq:8} corresponding to the initial values of $p(t_0)=q(t_0)$ and $\dot{p}(t_0)=\dot{\kappa}(0)\dot{q}(t_0)$ (or $\dot{p}(t_0)=\dot{q}(t_0)$ if $\kappa\in\mathcal{K}_1$).

\section{Main Results}
\label{sec:III}

This section presents the main results of the paper. Theorems \ref{the:1} and \ref{the:2} discuss the PTC for the nominal and perturbed systems, respectively. Theorem \ref{th:3} presents an output assessment rule for a system under PTC. First, consider the following lemma, which is used for the proof of the subsequent results.
\begin{lemma}
\label{lem:2}
Suppose there exist a class $\mathcal{M}$ function $\mu(t)$ and $\tilde{t}\in[t_0,\infty)$ such that a vector valued function $r(\cdot):[t_0,\infty)\rightarrow\mathbb{R}^n$ satisfies the following condition:
\begin{equation}
\label{eq:7p}
\|r(t)\|\leq\dot{\mu}(t), \forall t\in[\tilde{t},\infty)
\end{equation}
Then, for any $\alpha>0$ and $\tau\in(0,\infty)$ there exists $\eta\in\mathcal{M}(\tau)$ such that the following conditions are satisfied:
\begin{equation}
\label{eq:7pp}
\lim_{t\rightarrow\infty}-\frac{\ddot{\eta}(t)}{\dot{\eta}^\alpha(t)}\|r(t)\|=0
\end{equation}
\begin{equation}
\label{eq:7ppp}
\lim_{t\rightarrow\infty}\frac{1}{\dot{\eta}^\alpha(t)}\|r(t)\|=0
\end{equation}
\end{lemma}
\begin{proof}
In this proof we use the third item of Lemma \ref{lem:1}. Suppose $\mathcal{M}\ni\mu(t)=1/(\alpha+1)(-\dot{\eta}^{\alpha+1}(t)+\dot{\eta}^{\alpha+1}(0))$. By differentiation we have $\dot{\mu}=-\ddot{\eta}\dot{\eta}^\alpha$. Then, equation \eqref{eq:7p} can be written as $\|r(t)\|\leq-\ddot{\eta}\dot{\eta}^\alpha$ and can be rewritten as $-\ddot{\eta}(t)/\dot{\eta}^\alpha(t)\|r(t)\|\leq\ddot{\eta}^2(t)$. Since $\lim_{t\rightarrow\infty}\ddot{\eta}^2(t)=0$, equality \eqref{eq:7pp} holds and the first part is proved. For the second part, again consider $\|r(t)\|\leq-\ddot{\eta}\dot{\eta}^\alpha\Rightarrow\|r(t)\|/\dot{\eta}^\alpha\leq-\ddot{\eta}$. Since $\lim_{t\rightarrow\infty}-\ddot{\eta}(t)=0$, the second result is also proved.
\end{proof}

\begin{assumption}
\label{ass:1}
Let $t_0$ denote the initial time and $\Delta t=t-t_0$. Suppose there exist $\mu\in\mathcal{M}(\tau)$ and $\tilde{t}>t_0$ such that for any $\tau>0$ the closed-loop response of system \eqref{eq:1} under ITC $u=f(\dot{q},q)$ satisfies $\|\dot{q}(t)\|\leq\dot{\mu}(\Delta t)$ for all $t\in[\tilde{t},\infty)$ and at least one of the following conditions hold:
\begin{enumerate}
\item $\|\ddot{q}(t)\|\leq\dot{\mu}^2(\Delta t), \forall t\in[\tilde{t},\infty)$. 
\item $\|f(\dot{q},q)\|\leq\dot{\mu}^2(\Delta t), \forall t\in[\tilde{t},\infty)$. 
\end{enumerate}
\end{assumption}

Always there exist class $\mathcal{M}$ functions to satisfy Assumption \ref{ass:1} if the infinite-time system is asymptotically stable since the growth rate of the derivative of a class $\mathcal{M}$ function can be tuned. Note that the inequalities of Assumption \ref{ass:1}, such as $\|\dot{q}(t)\|\leq\dot{\mu}(\Delta t)$, are needed to be satisfied for the final times, i.e., $t>\tilde{t}$, and they are not required to be held for the transition interval, i.e., $t<\tilde{t}$. The next lemma is followed by a proposition to analytically formulate an appropriate class $\mathcal{M}$ function when an exponentially stable ITC is used. 
 
 \begin{lemma}[\cite{Godunov}]
\label{lem:A-1}
For a real matrix $Q\in\mathbb{M}^n$ we have $\|\exp(Qt)\|\leq\sqrt{2\|X^{-1}\|\|X\|}\exp(-t/(2\|X\|))$ where $X\in\mathbb{M}^n$ is the solution of the following Lyapunov equation:
\begin{equation}
\label{eq:A-2}
XQ+Q^TX=-\mathbb{I}_n
\end{equation}
\end{lemma}

\begin{prop}
Suppose that the closed-loop system under an ITC $u=f(\dot{q},q)$ is exponentially stable such that for some $\varrho>0$, $\rho>0$, and negative-definite matrix $Q$, we have $\|[q^T(t)-q^T_d,\dot{q}^T(t)]\|\leq\varrho\|\exp(Qt)\|$ and $\|\ddot{q}(t)\|\leq\rho\|\exp(Qt)\|$ for all $t\in[0,\infty)$. Then, a class $\mathcal{M}(\tau)$ function with the following derivative satisfies Assumption \ref{ass:1}:
\begin{equation}
\label{eq:A-7}
\mu(\Delta t)=\tau\left[1-\exp\left(-\alpha\Delta t /\|X\|\right)\right], \quad 0<\alpha<0.5
\end{equation}
where $X$ is the solution of the Lyapunov equation \eqref{eq:A-2}.
\end{prop}
\begin{proof}
Use Lemma \ref{lem:1}--item 6.
\end{proof}
 
The following theorem uses the previous results to obtain a PTC for system \eqref{eq:1}.

\begin{theorem}[Nominal system]
\label{the:1}
Let $\kappa\in\mathcal{K}_1(\tau)$ be the inverse function of $\mu\in\mathcal{M}_1(\tau)$ and function $h(\dot{q},q,t)$ be defined as follows:
\begin{equation}
\label{eq:19}
h(\dot{q},q,t)=\left\{\begin{array}{ll}\dot{\kappa}^2(\Delta t)f\left(\dot{q}/\dot{\kappa}(\Delta t),q\right) & \\
+\left[\ddot{\kappa}(\Delta t)/\dot{\kappa}(\Delta t)\right]M(q)\dot{q} & \\
+\left[1-\dot{\kappa}^2(\Delta t)\right]g(q) & \hspace{2mm} t\in[t_0,t_0+\tau) \vspace{2mm}\\ 
f\left(\dot{q},q\right) & \hspace{2mm} t\in[t_0+\tau,\infty)
\end{array}\right.
\end{equation}
Then, the closed-loop system constructed by the Euler-Lagrange equation \eqref{eq:1} and the PTC $u=h(\dot{q},q,t)$:
\begin{enumerate}
\item is uniformly stable, if system \eqref{eq:1} under $u=f(\dot{q},q)$ is stable and $\mu$ is selected such that Assumption \ref{ass:1} is satisfied.
\item is (globally) prescribed-time attractive and converges at $t=t_0+\tau$, if system \eqref{eq:1} under $u=f(\dot{q},q)$ is (globally) attractive and $\mu$ is selected such that Assumption \ref{ass:1} is satisfied.
\item is (globally) uniformly prescribed-time stable and converges at $t=t_0+\tau$, if system \eqref{eq:1} under $u=f(\dot{q},q)$ is (globally) asymptotically stable and $\mu$ is selected such that Assumption \ref{ass:1} is satisfied.
\item is (globally) uniformly prescribed-time stable and converges at $t=t_0+\tau$, if system \eqref{eq:1} under $u=f(\dot{q},q)$ is (globally) exponentially stable and $\mu$ satisfies \eqref{eq:A-7}.
\item avoids obstacles, if $u=f(\dot{q},q)$ is an obstacle avoidance controller. Generally, the system under the proposed PTC goes the same path but faster, within the specified time interval.
\end{enumerate}
\end{theorem}
\begin{proof}
First, we prove item 1 about Lyapunov stability. We concern about $\dot{q}(t)$ because $q(t)$ is the same with both infinite and prescribed-time controllers. Recall from the mapping rules of Subsection \ref{subsec:II-B} that if $\dot{q}(t)$ be the velocity of the infinite-time system, then the velocity of the prescribed-time system is $\dot{q}(t)/\dot{\mu}(\Delta t)$ (with the same initial conditions since $\kappa(t)$ is class $\mathcal{K}_1$ which means that $\dot{\kappa}(0)=1$). If the infinite-time system is stable, then for every $\varepsilon>0$ there exists $\xi(\varepsilon)>0$ such that $\|[q^T(t_0),\dot{q}^T(t_0)]\|<\xi(\varepsilon)$ implies $\|[q^T(t),\dot{q}^T(t)]\|<\varepsilon$. Consequently, the Lyapunov stability is proved for the prescribed-time system if for every $\sigma>0$, there exists $\delta(\sigma,t_0)>0$ such that $\|[q^T(t_0),\dot{q}^T(t_0)]\|<\delta(\sigma,t_0)$ implies $\|[q^T(t),\dot{q}^T(t)/\dot{\mu}(\Delta t)]\|<\sigma$ and the uniformity is then shown by proving the independency of $\delta(\sigma,t_0)$ to $t_0$. Accordingly, define the following function associated with the infinite-time response of system \eqref{eq:1}:
\begin{equation}
\label{eq:pf1}
c_1(\delta,t,t_0)=\underset{\left\|\left[{q(t_0)\atop \dot{q}(t_0)}\right]\right\|<\delta}{\mathrm{sup}}(\|\dot{q}(t,q(t_0),\dot{q}(t_0))\|)
\end{equation}
As the infinite-time system is autonomous, $c_1(\delta,t,t_0)$ is independent of $t_0$, thus we write $c_1(\delta,\Delta t)\equiv c_1(\delta,t,t_0)$. Note that since the infinite-time system is stable, for any $\bar{c}_1>0$, there exists a $\delta>0$ such that $c_1(\delta,\Delta t)<\bar{c}_1$. Besides, define the following function:
\begin{equation}
\label{eq:pf1}
c_2(\delta)=\underset{\Delta t\in[0,\infty)}{\mathrm{sup}}\left(c_1(\delta,\Delta t)/\dot{\mu}(\Delta t)\right)
\end{equation}
According to \eqref{eq:pf1}, one can write $\|\dot{q}(t)\|\leq c_2(\delta)\dot{\mu}(\Delta t)$. Note that for any $\bar{c}_2>0$, there exists a $\delta>0$ such that $c_2(\delta)<\bar{c}_2$, since $c_1(\delta,\Delta t)$ can be set arbitrarily small as $c_1(\delta,\Delta t)<\bar{c}_1$ for any $\bar{c}_1$. Therefore, an arbitrary bound on $\|\dot{q}(t)/\dot{\mu}(\Delta t)\|$ corresponds to a bound on the initial conditions, and because the same claim is true for $q(t)$, we conclude that the mapped system is stable. Moreover, the uniformity in time is deduced from the fact that $c_2(\delta)$, and thus $\delta$, are independent of $t_0$.

For the rest of the items we prove the prescribed-time attractivity condition. System \eqref{eq:1} can be mapped onto a finite interval, similar to the mapping strategy of Subsection \ref{subsec:II-B}. The mapped version of system \eqref{eq:1} can be stated as the following form after substituting $p$ by $q$ and $\eta$ by $t$: 
\begin{equation}
\label{eq:21}
\begin{split}
&M(q)\ddot{q}+C(\dot{q},q)\dot{q}+\dot{\kappa}^2(\Delta t)g(q)-\frac{\ddot{\kappa}(\Delta t)}{\dot{\kappa}(\Delta t)}M(q)\dot{q}\\
&=\dot{\kappa}^2(\Delta t)f(\dot{q}/\dot{\kappa}(\Delta t),q)
\end{split}
\end{equation}
The above equation is achieved by a one-to-one mapping of an attractive infinite-time closed-loop solution. At this step, we seek a $u=h(\dot{q},q,t)$ for system \eqref{eq:1} such that the closed-loop system behaves similar to \eqref{eq:21}. It can be shown by a simple substitution that the goal is achievable by the use of \eqref{eq:19} for $t\in[t_0,t_0+\tau)$. However, for $t\in[t_0+\tau,\infty)$, we need to prove that the closed-loop response does not change the state of the system. According to the form of system \eqref{eq:1} and the control law stated in \eqref{eq:19}, it is sufficient to prove $\lim_{(\dot{q},q,\Delta t)\rightarrow(0,q_d,\tau^-)}h(\dot{q},q,t)=g(q_d)$, i.e., $h(\dot{q},q,t)$ is continuous. Since $u=f(\dot{q},q)$ stabilizes system \eqref{eq:1}, $\lim_{(\dot{q},q)\rightarrow(0,q_d)}f(\dot{q},q)=g(q_d)$, therefore $\lim_{(\dot{q},q,\Delta t)\rightarrow(0,q_d,\tau^-)}h(\dot{q},q,t)=\lim_{(\dot{q},\Delta t)\rightarrow(0,\tau^-)}[\ddot{\kappa}(\Delta t)/\dot{\kappa}(\Delta t)]M(q)\dot{q}+g(q_d)$. Also, it is provable that the value of $\lim_{(\dot{q},\Delta t)\rightarrow(0,\tau^-)}(\ddot{\kappa}/\dot{\kappa})\dot{q}$ in a closed-loop system with $u=h(\dot{q},q,t)$ is equal to the value of $\lim_{t\rightarrow\infty}-(\ddot{\mu}/\dot{\mu}^3)\dot{q}$ in a closed-loop system with $u=f(\dot{q},q)$ (the argument of $\kappa$ and its derivatives is $\Delta t$). According to Lemma \ref{lem:2}, $\lim_{(\dot{q},q,\Delta t)\rightarrow(0,q_d,\tau^-)}h(\dot{q},q,t)=g(q_d)$ as long as $\dot{q}(t)$ is bounded by a $\dot{\mu}$ at infinity. The boundedness of $\|h(\dot{q},q,t)\|$ is provable according to the boundedness of $(\ddot{\kappa}/\dot{\kappa})\|\dot{q}\|$ and $\dot{\kappa}^2\|f(\dot{q},q)\|$ (which the latter is equivalent to the boundedness of $\dot{\kappa}^2\|\ddot{q}\|$) since all the terms expressed by \eqref{eq:19} are bounded after satisfaction of Assumption \ref{ass:1}. 
\end{proof}

The following corollary presents a special form of the control law discussed in Theorem \ref{the:1} for PD controllers with gravity compensation:

\begin{corollary}[PD with gravity compensation]
\label{cor:1}
There exists a $\kappa\in\mathcal{K}(\tau)$ such that the following gain scheduling PD controller with gravity compensation is a globally uniformly prescribed-time stable controller for system \eqref{eq:1}, which converges in a prescribed time $\tau$:
\begin{equation}
\label{eq:24}
u=\tilde{P}(t)q_e+\tilde{D}(t)\dot{q}+g(q)
\end{equation}
where $q_e=q-q_d$ and
\begin{equation}
\label{eq:25}
\tilde{P}(t)=\left\{\begin{array}{ll}\dot{\kappa}^2(\Delta t)P\\
P\end{array}\hspace{2mm}\begin{array}{ll}t\in[t_0,t_0+\tau)\\
t\in[t_0+\tau,\infty)\end{array}\right.
\end{equation}
\begin{equation}
\label{eq:26}
\resizebox{1\columnwidth}{!}{
$\tilde{D}(t)=\left\{\begin{array}{ll}\dot{\kappa}(\Delta t)D+[\ddot{\kappa}(\Delta t)/\dot{\kappa}(\Delta t)]M(q)\\
D\end{array}\hspace{2mm}\begin{array}{ll}t\in[t_0,t_0+\tau)\\
t\in[t_0+\tau,\infty)\end{array}\right.$}
\end{equation}
if negative-definite $P,D\in\mathbb{M}^n$ satisfy Assumption \ref{ass:1}. 
\end{corollary}
\begin{proof}
Consider the ITC as $f(\dot{q},q)=P(q-q_d)+D\dot{q}+g(q)$, $P,D<0$, which is an asymptotically stable controller for system \eqref{eq:1} in $q=q_d$ and use Theorem \ref{the:1}--item 3.
\end{proof}

In the following corollary, there is no need to satisfy Assumption \ref{ass:1}.

\begin{corollary}[Feedback linearization]
\label{cor:2}
Let $P,D\in\mathbb{M}^n$ be arbitrary negative definite matrices. The following feedback linearization controller is a globally uniformly prescribed-time stable controller for system \eqref{eq:1}, which converges in a prescribed time $\tau$:
\begin{equation}
\label{eq:A-3}
u=C(\dot{q},q)\dot{q}+g(q)+M(q)\left(\tilde{P}(t)q_e+\tilde{D}(t)\dot{q}\right)
\end{equation}
where $q_e=q-q_d$ and
\begin{equation}
\label{eq:25_A}
\tilde{P}(t)=\left\{\begin{array}{ll}\dot{\kappa}^2(\Delta t)P\\
P\end{array}\hspace{2mm}\begin{array}{ll}t\in[t_0,t_0+\tau)\\
t\in[t_0+\tau,\infty)\end{array}\right.
\end{equation}
\begin{equation}
\label{eq:26_A}
\tilde{D}(t)=\left\{\begin{array}{ll}\dot{\kappa}(\Delta t)D+\ddot{\kappa}(\Delta t)/\dot{\kappa}(\Delta t)\mathbb{I}_n\\
D\end{array}\hspace{2mm}\begin{array}{ll}t\in[t_0,t_0+\tau)\\
t\in[t_0+\tau,\infty)\end{array}\right.
\end{equation}
such that $\mu=\kappa^{-1}$ satisfies \eqref{eq:A-7} in which $X$ is calculated from \eqref{eq:A-2} where $Q$ is
\begin{equation}
Q=\begin{bmatrix}
0 & \mathbb{I}_n\\
P & D
\end{bmatrix}
\end{equation}
\end{corollary}
\begin{proof}
Use the ITC $u=f(\dot{q},q)=C(\dot{q},q)\dot{q}+g(q)+M(q)(Pq_e+D\dot{q})$ and implement Theorem \ref{the:1}--item 4.
\end{proof}

\begin{theorem}[Output assessment]
\label{th:3}
Suppose $t_0=0$, let $V(\cdot,\cdot,\cdot),W(\cdot,\cdot,\cdot):\mathbb{R}^{n}\times\mathbb{R}^{n}\times[0,\infty)\rightarrow\mathbb{R}^{m}$ be two vector-valued functions, and $\kappa\in\mathcal{K}_1(\tau)$ be the inverse function of $\mu\in\mathcal{M}_1(\tau)$. Then, an output $V(\dot{q},q,\mu)\coloneqq W(\dot{q}/\dot{\kappa}(\mu),q,\kappa(\mu))$ of system \eqref{eq:1} under the PTC $u=h(\dot{q},q,t)$ at $\mu$, is equal to an output $W(\dot{q},q,t)=V(\dot{q}/\dot{\mu}(t),q,\mu(t))$ of the same system under the ITC $u=f(\dot{q},q,t)$ at $t=\kappa(\mu)$. Moreover, the equalities hold for the sign of their derivative with respect to their own time scales.
\end{theorem}
\begin{proof}
Consider the fact that the closed-loop dynamics under the proposed PTC is the mapped version of the same system under the ITC, which is used as the core of the PTC formulation. Recall from the rules of the mapping technique discussed in Subsection \ref{subsec:II-B} that the position vectors are the same in infinite-time and prescribed-time systems, but if the velocity of the infinite-time system at $t$ is shown by $\dot{q}(t)$, then the velocity of the mapped prescribed-time system at $\mu(t)$ is $\dot{q}(t)/\dot{\mu}(t)$. The claim is proved by substituting the identical values. For the derivatives, consider the fact that sign of $dV/dt=\lim_{\varsigma\rightarrow0}(V(t+\varsigma)-V(t))/\varsigma$ does not change as long as sign of $V(t+\varsigma)-V(t)$ remains unchanged.
\end{proof}

\begin{corollary}[Prescribed-time Lyapunov function]
\label{rem:bound}
Suppose $t_0=0$ and $W(\dot{q},q,t)$ is a Lyapunov function for system \eqref{eq:1} under an ITC. Then, $V(\dot{q},q,\mu)=W(\dot{q}/\dot{\kappa}(\mu),q,\kappa(\mu))$ is a Lyapunov function for the corresponding PTC.
\end{corollary}

\begin{remark}
\label{rem:bound}
According to Theorem \ref{th:3}, the controller magnitude can be considered as an output and analyzed in the infinite time domain to assess its behavior and bounds before implementation. Accordingly, assuming $t_0=0$, the proposed PTC in the infinite-time domain is $l(\dot{q}(t),q(t),t)=h(\dot{q}(\mu)/\dot{\kappa}(\mu),q(\mu),\kappa(\mu))=f(\dot{q},q)/\dot{\mu}^2-(\ddot{\mu}/\dot{\mu}^3)M(q)\dot{q}+(1-1/\dot{\mu}^2)g(q)$. Let us define the normalized control inputs as $\tilde{f}(\dot{q},q)=f(\dot{q},q)-g(q)$ and similarly, $\tilde{l}(\dot{q},q,t)=l(\dot{q},q,t)-g(q)$. The magnitude of the normalized PTC is then $\|\tilde{l}(\dot{q},q,t)\|=\|\tilde{f}(\dot{q},q)/\dot{\mu}^2-(\ddot{\mu}/\dot{\mu}^3)M(q)\dot{q}\|\leq\|\tilde{f}(\dot{q},q)\|/\dot{\mu}^2-(\ddot{\mu}/\dot{\mu}^3)\|M(q)\|\cdot\|\dot{q}\|$. Define $\mathcal{M}^{\prime}=\{\mu\in\mathcal{M}:\|\tilde{f}(\dot{q},q)\|<\dot{\mu}^2\}$ and $\mathcal{M}^{\prime\prime}=\{\mu\in\mathcal{M}:\|\dot{q}\|<-\ddot{\mu}/\dot{\mu}^3\}$. If $\mu\in\mathcal{M}^{\prime}\cap\mathcal{M}^{\prime\prime}$, then $\|\tilde{l}(\dot{q},q,t)\|\leq\delta(1+\|M(q)\|)$ for some positive $\delta<1$. The bound of the normalized controller is then dependent on the maximum value of $\|M(q)\|$ which is specifiable (maybe independent of the input signal) for the system.
\end{remark}

The control technique proposed by Theorem \ref{the:1} (and its following results) is proved to be usable for the nominal system; however, the following theorem proves that the same controller can be used for perturbed systems. 

\begin{theorem}[Perturbed system]
\label{the:2}
Suppose $t_0=0$ and $d(t)$ satisfies Assumption \ref{ass:d}. Let $\kappa\in\mathcal{K}_1(\tau)$ be the inverse function of $\mu\in\mathcal{M}_1(\tau)$, and function $h(\dot{q},q,t)$ defined as \eqref{eq:19}. Then, the perturbed system \eqref{eq:2} under the PTC $u=h(\dot{q},q,t)$ is (globally) prescribed-time attractive and converges at $t=\tau$, if the unperturbed system \eqref{eq:1} under $u=h(\dot{q},q,t)$ is (globally) prescribed-time attractive (see Theorem \ref{the:1}.) 
\end{theorem}
\begin{proof}
The closed-loop response of system \eqref{eq:2} under $u=h(\dot{q},q,t)$ is
\begin{eqnarray}\nonumber
\label{eq:27}
M(q)\ddot{q}+C(\dot{q},q)\dot{q}+\dot{\kappa}^2(t)g(q)-\frac{\ddot{\kappa}(t)}{\dot{\kappa}(t)}M(q)\dot{q}\\
=d(t)+\dot{\kappa}^2(t)f(\dot{q}/\dot{\kappa}(t),q)
\end{eqnarray}
To simplify the analysis, consider the reverse-mapping of \eqref{eq:27} in a conventional infinite-time domain: 
\begin{equation}
\label{eq:28}
M(q)\ddot{q}+C(\dot{q},q)\dot{q}+g(q)
=f(\dot{q},q)+\dot{\mu}^2d(\mu(t))
\end{equation}
Since Assumption \ref{ass:d} is satisfied, $\lim_{t\rightarrow\infty}\dot{\mu}^2d(\mu(t))=0$, the additional term in the right-hand side of \eqref{eq:28} vanishes, and the system acts same as the closed-loop response of system \eqref{eq:1} under the ITC at infinity.
\end{proof}
\begin{remark}
\label{rem:2}
Theorem \ref{the:2} states that the ITC $u=f(\dot{q},q)$ can be designed regardless of the disturbances. This is a useful feature for the PTCs as the designer does not need to consider an upper bound for the disturbances. However, the existence of disturbances may inevitably cause large variations in the control input. In fact, according to the disturbance rejection feature of the proposed controller, it also can be considered as a model-free online control scheme if the system model is entirely considered as a disturbance term. 
\end{remark}

\begin{remark}
\label{rem:gain}
It is provable that a controller, viewed as a gain multiplied by the state vector, cannot reach zero equilibrium at a (known or unknown) finite time unless the gain approaches infinity, which also is the case for prescribed-time controllers. As a result, the proposed controller of Theorem \ref{the:1} (and Corollaries \ref{cor:1} and \ref{cor:2}), may not be practically implementable in robotic systems in its presented form according to the unbounded divergence of $\dot{\kappa}(\Delta t)$ and $\ddot{\kappa}(\Delta t)/\dot{\kappa}(\Delta t)$ when limiting to $t=t_0+\tau$ which imposes problems in terms of storing the gain values on the robot's memory. To solve this problem, according to \cite{zhou2021prescribed}, one may cease the increment of gains by switching at $t=t_0+\tau-\epsilon$ to a new controller with finite gain values. We propose the following switching controller\footnote{MATLAB\textsuperscript{\tiny\textregistered} codes and Simulink\textsuperscript{\tiny\textregistered} models for the proposed controller can be found in \href{https://github.com/a-shakouri/prescribed-time-control}{https://github.com/a-shakouri/prescribed-time-control}} as a combination of state and time-dependent switching rules by which the coefficients remain bounded for all $t\in[t_0,\infty)$:
\begin{equation}
\label{eq:26p}
h(\dot{q},q,t)=\left\{\begin{array}{ll}
\dot{\kappa}^2(\Delta t)f\left(\dot{q}/\dot{\kappa}(\Delta t),q\right) & \\
+\left[\ddot{\kappa}(\Delta t)/\dot{\kappa}(\Delta t)\right]M(q)\dot{q} & \\
+\left[1-\dot{\kappa}^2(\Delta t)\right]g(q) &\hspace{2mm}(q,\dot{q},t)\in\mathcal{S} \vspace{2mm}\\
\dot{\kappa}^2(\Delta t_s)f\left(\dot{q}/\dot{\kappa}(\Delta t_s),q\right) & \\
+\left[\ddot{\kappa}(\Delta t_s)/\dot{\kappa}(\Delta t_s)\right]M(q)\dot{q} & \\
+\left[1-\dot{\kappa}^2(\Delta t_s)\right]g(q) &\hspace{2mm}(q,\dot{q},t)\notin\mathcal{S}
\end{array}\right.
\end{equation}
where the switching rule is defined by $\mathcal{S}$ as a set-valued function of user-defined constants $\epsilon>0$ and $\sigma>0$:
\begin{equation}
\label{eq:26pp}
\mathcal{S}(\epsilon,\sigma)=\{(q,\dot{q},t):t\leq t_0+\tau-\epsilon,\|[q^T,\dot{q}^T]\|\geq\sigma\}
\end{equation}
and $\Delta t_s=t_s-t_0$ with $t_s$ standing for the switching time. In a merely time-dependent switching, we have $\Delta t_s=\tau-\epsilon$, otherwise $t_s$ should be obtained and saved when the switching occurs. Note that the state-dependent switching automatically makes the system reach an arbitrary nonzero error bound $\sigma>0$ before $t=t_0+\tau$, while the time-dependent switching avoids the gains to violate the limits of the system's memory.
\end{remark}

The following algorithm summarizes the proposed PTC design method:

\begin{alg}[PTC design]
\label{al:1}
The input of this algorithm is system \eqref{eq:2} with unknown $d(t)$ satisfying Assumption \ref{ass:d}. The output is a (globally) prescribed-time attractive controller $u=h(\dot{q},q,t)$ for system \eqref{eq:2} such that the convergence occurs at $t=\tau$, the disturbance $d(t)$ is rejected before reaching $t=\tau$, and the system avoids collision with fixed obstacles. In this regard, select a desired time of convergence $\tau$. Then: 
\begin{enumerate}
\item Design a controller in order to asymptotically stabilize the closed-loop response of system \eqref{eq:1} and call it $u=f(\dot{q},q)$ (as shown in Fig. \ref{fig:0}). 
\item If the system is exponentially stable, calculate $\mu\in\mathcal{M}_1(\tau)$ from condition \eqref{eq:A-7} and go to step 5.
\item Select a (different) $\mu\in\mathcal{M}_1(\tau)$ (with lower decay rate).
\item If Assumption \ref{ass:1} is satisfied, continue; otherwise, go back to step 3. 
\item Select a sufficiently small value for $\epsilon>0$ and $\sigma>0$ to avoid infinite gains and numerical problems. 
\item Substitute $f(\dot{q},q)$ in \eqref{eq:26p} with $\kappa=\mu^{-1}$ and obtain $u=h(\dot{q},q,t)$ (as shown in Fig. \ref{fig:0}). 
\end{enumerate}
\end{alg}

\begin{figure}[!h]
\centering\includegraphics[width=1\linewidth]{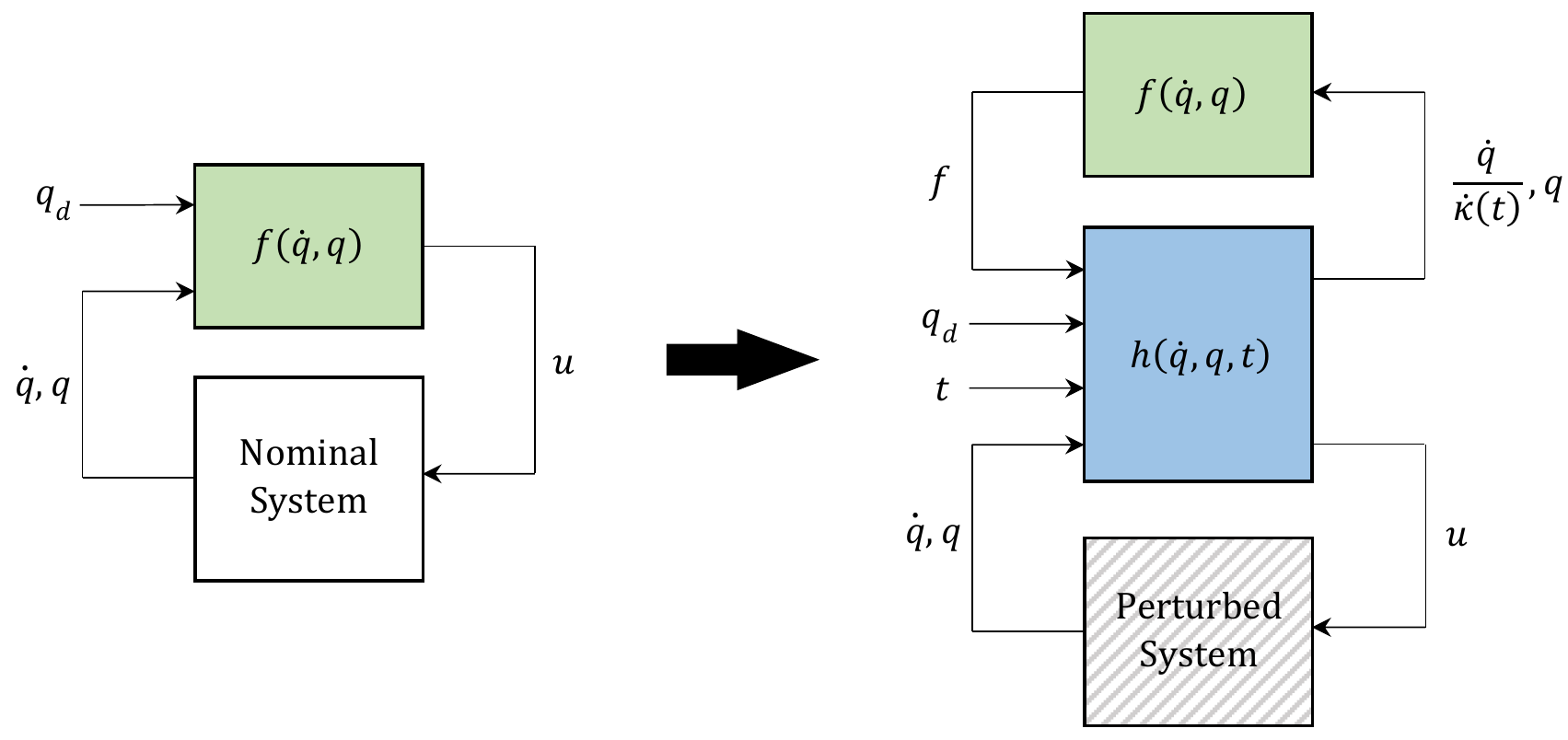}
\caption{Derivation of a PTC for a perturbed system from an ITC of the nominal system.}
\label{fig:0}
\end{figure}

\section{Examples}
\label{sec:IV}

In this section, some candidate class $\mathcal{K}$ functions are introduced and a two-link manipulator is controlled using the proposed scheme.  

\subsection{Mapping functions}
\label{sec:IV-A}

According to Definition \ref{def:1} and Lemma \ref{lem:1}--item 4, a class $\mathcal{K}$ function can be considered in the following forms for any $a_i>0$, $b_i>0$ and $c_i>0$:
\begin{equation}
\label{eq:29}
\kappa(t)=\textstyle\sum_{i=1}^{n}a_it^{b_i}/(\tau-t)^{c_i}
\end{equation}
\begin{equation}
\label{eq:30}
\kappa(t)=\textstyle-\sum_{i=1}^{n}a_i\ln(1-t/\tau)
\end{equation}
\begin{equation}
\label{eq:30p}
\kappa(t)=\textstyle\sum_{i=1}^{n}a_i\tan^{b_i}(\pi/2\cdot t/\tau)
\end{equation}

A class $\mathcal{M}$ function can be considered as the inverse function of some $\kappa(t)$. According to Lemma \ref{lem:1}--item 4, if $\mu_i(t)\in\mathcal{M}$, then $1/n\sum_{i=1}^{n}\mu_i(t)\in\mathcal{M}$. This kind of definition increases the flexibility of candidate functions and consequently, the conditions of Lemma \ref{lem:2} may be satisfied. 

\subsection{Two-link manipulator}
\label{sec:IV-B}

The 2-DOF manipulator shown in Fig. \ref{fig:1} is considered as the numerical example\footnote{MATLAB\textsuperscript{\tiny\textregistered} codes and Simulink\textsuperscript{\tiny\textregistered} models for this example can be found in \href{https://github.com/a-shakouri/prescribed-time-control-examples}{https://github.com/a-shakouri/prescribed-time-control-examples}}. The state variables and the control inputs are defined as $q=[q_1,\ q_2]^T$ in degrees, $\dot{q}=[\dot{q}_1,\ \dot{q}_2]^T$ in degrees per second, and $u=[u_1,\ u_2]^T$ in Newton-meters, respectively. This manipulator can be modeled in the way defined by system \eqref{eq:1} where $M(q)$, $C(\dot{q},q)$, and $g(q)$ are obtained from \cite{spong} and the parameters are considered as $l_1=l_2=1\ \mathrm{m}$ for the length of links, $l_{c1}=l_{c2}=0.5\ \mathrm{m}$ for the center of mass positions, $m_1=m_2=1\ \mathrm{kg}$ for the mass of the links, $I_1=I_2=0.33\ \mathrm{kg\cdot m^2}$ for the moments of inertia, and $g_0=9.81\ \mathrm{m/s^2}$ for the gravitational acceleration:
\begin{equation}
\label{eq:33_1}
M_{11}(q)=\cos(q_2)+2.16
\end{equation}
\begin{equation}
\label{eq:33_2}
\begin{split}
M_{12}(q)=M_{21}(q)&=0.5\sin(q_1+q_2)\sin(q_1)\\
&+0.25\sin^2(q_1+q_2)+0.33
\end{split}
\end{equation}
\begin{equation}
\label{eq:33_3}
M_{22}(q)=0.25\sin^2(q_1+q_2)+0.33
\end{equation}
\begin{equation}
\label{eq:34}
C(q,\dot{q})=-\cos(q_1+q_2)\begin{bmatrix}
\dot{q}_2 & \dot{q}_1+\dot{q}_2 \\
-\dot{q}_1 & 0
\end{bmatrix}
\end{equation}
\begin{equation}
\label{eq:35}
g(q)=g_0\begin{bmatrix}
1.5\cos(q_1)+0.5\cos(q_1+q_2) \\
0.5\cos(q_1+q_2)
\end{bmatrix}
\end{equation}
The initial state is $q=\dot{q}=0$ at $t_0=0$ and the desired position to be reached is $q_d=[90^\circ,0]^T$ that is an unstable equilibrium point for the open-loop system.

\begin{figure}[!h]
\centering\includegraphics[width=0.4\linewidth]{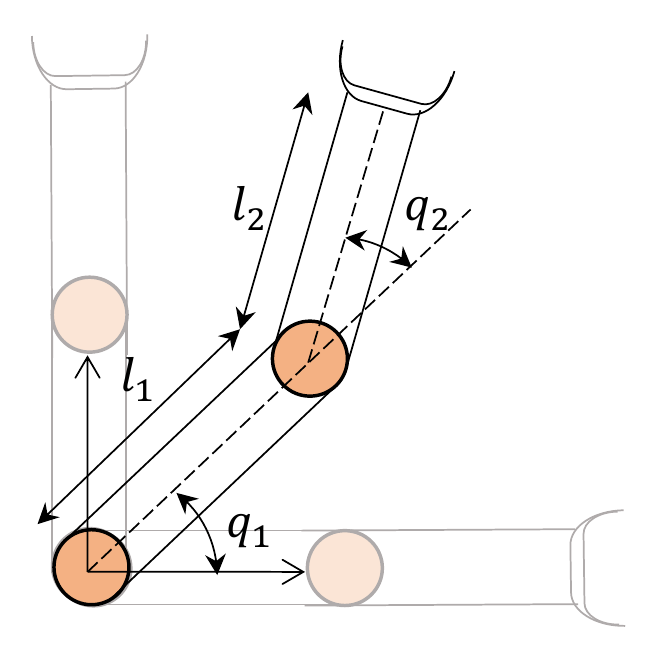}
\caption{A schematic view for a 2-DOF manipulator with revolute joints.}
\label{fig:1}
\end{figure}

Consider a joint limit avoidance PD controller with gravity compensation in the following form:

\begin{equation}
\label{eq:36}
f(\dot{q},q)=Pq_e+D\dot{q}+g(q)+\gamma_q(q)
\end{equation}
where $q_e=q-q_d$, $P=-0.1\mathbb{I}_2$, $D=-\mathbb{I}_2$, and $\gamma_q=[{\gamma_q}_1,\ {\gamma_q}_2]^T$ is the acceleration vector that prevent the manipulator from violating its joint limits. For the first joint no limits are considered and for the second joint it is assumed that the operation bound should be within $-3^\circ$ to $3^\circ$, and the distance limit of the potential influence is considered to be $0.5^\circ$. Therefore, ${\gamma_q}_1=0$, and ${\gamma_q}_2$ can be formulated as \cite{khatib}:
\begin{equation}
\label{eq:37}
{\gamma_q}_2(q)=\left\{
\begin{array}{ll}
\left(\frac{1}{q_2+3^\circ}-\frac{1}{0.5^\circ}\right)\frac{10^{-9}}{(q_2+3^\circ)^2} & \quad q_2<-2.5^\circ \\
\left(\frac{1}{3^\circ-q_2}-\frac{1}{0.5^\circ}\right)\frac{-10^{-9}}{(3^\circ-q_2)^2} & \quad q_2>2.5^\circ \\
0 & \quad |q_2|\leq2.5^\circ
\end{array}\right.
\end{equation}
The corresponding PTC is designed using Algorithm \ref{al:1} with $\epsilon=1\textrm{ s}$ and a mapping function defined by \eqref{eq:29} with $n=1$, $a_1=20$, $b_1=c_1=1$, and $\tau=20\ \mathrm{s}$ (such that $\dot{\kappa}(0)=1$ is satisfied, thus we have a class $\mathcal{K}_1$ function). 

Suppose the 2-DOF manipulator dynamics is defected by a disturbed acceleration $d(t)\in\mathbb{R}^2$ such that $d(t)=\int_0^td^\prime(t)dt$ is a Wiener process where $d^\prime(t)$ is a zero-mean normally distributed vector with a standard deviation of $0.1\ \mathrm{N\cdot m}$. The disturbances are assumed considerably large in order to exaggerate what is happening in the system. 

Fig. \ref{fig:2} compares the closed-loop response of the proposed PTC, $h(\dot{q},q,t)$, with the core ITC, $f(\dot{q},q)$, formulated in \eqref{eq:36}. Several simulations are carried out for the perturbed case to demonstrate the error bounds. As it is shown, the ITC (plotted in black and gray) is almost useless for the perturbed case, while the PTC (plotted in dark and light blue) perfectly performs its task and converges at $\tau=20\ \mathrm{s}$.

\begin{figure*}[!h]
\centering\includegraphics[width=0.9\linewidth]{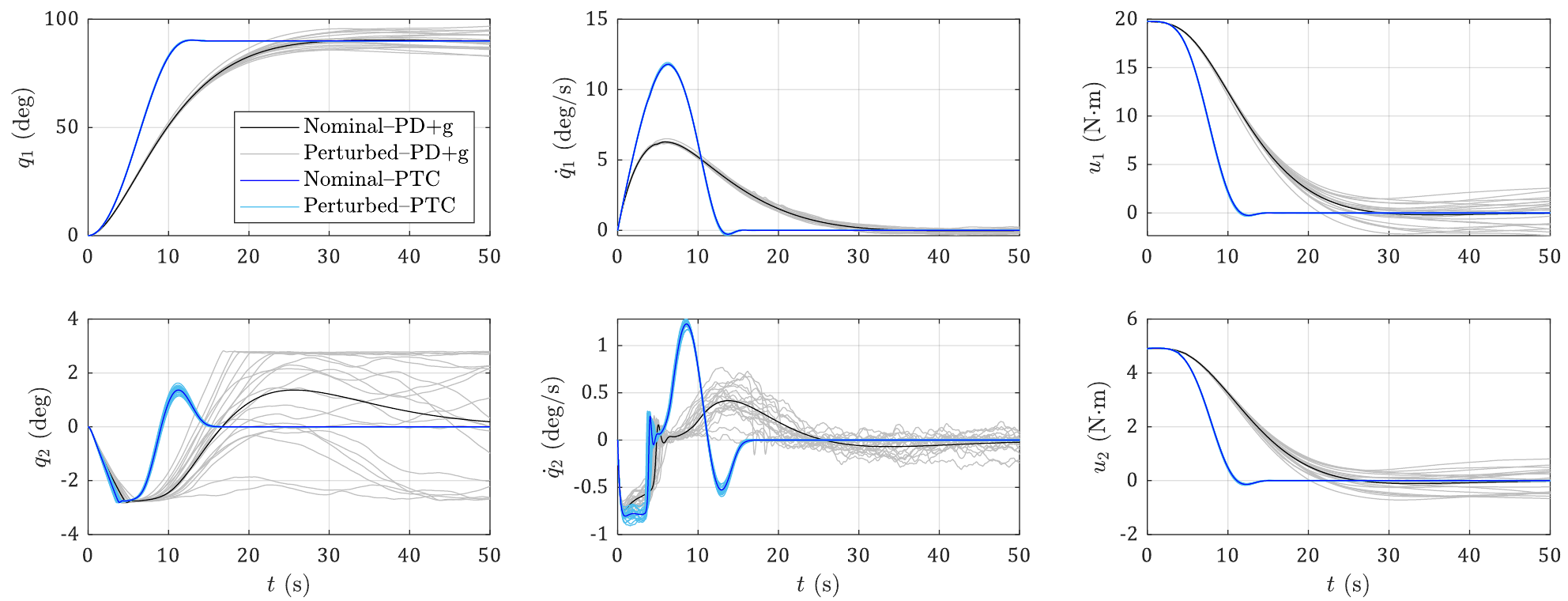}
\caption{A 2-DOF robot under PD controller with gravity compensation (PD+g) and prescribed-time controller (PTC) with $\tau=20\hspace{1mm}\mathrm{s}$}
\label{fig:2}
\end{figure*}

\section{Conclusions}
\label{sec:V}

A class of prescribed-time controllers (PTCs) has been introduced and analyzed in this paper capable of actively rejecting disturbances with unknown bounds without observing the disturbances. It has been shown that a stable PTC with an arbitrary convergence time can be designed just by substituting a conventional infinite-time controller (ITC) into a time-dependent formula. It has been proved that if the ITC is asymptotically stable, then the obtained PTC is prescribed-time stable. To expand the applications of the proposed approach, many ITCs can be incorporated with the presented method to provide their infinite-time properties in a prescribed time window. 


\bibliographystyle{IEEEtran}
\bibliography{root}

\end{document}